\documentclass[leqno,11pt]{amsart}

\usepackage{amssymb}
\usepackage{color}
\usepackage{latexsym}
\usepackage{amscd}
\theoremstyle{plain}
\newtheorem{theo}{Theorem}[section]
\newtheorem{coro}[theo]{Corollary}
\newtheorem{prop}[theo]{Proposition}
\newtheorem{lem}[theo]{Lemma}
\newtheorem{exa}[theo]{Example}
\theoremstyle{definition}
\newtheorem{rem}[theo]{Remark}
\newtheorem{df}[theo]{Definition}
\newenvironment{pf}{{\noindent\bf Proof. }}{\hfill $\square$\medskip}
\title[Chern-Moser operators]{Chern-Moser operators and weighted jet determination problems in higher codimension}
\author{L\'ea Blanc-Centi and Francine Meylan}

\def\1#1{\ov{#1}}
\def\2#1{\widetilde{#1}}
\def\5#1{\mathfrak{#1}}
\def\6#1{\mathcal{#1}}
\def\4#1{\mathbb{#1}}
\def\3#1{\widehat{#1}}

\def\C{{\4C}}

\newcounter{remarkcislo}[section]
\newcounter{corcislo}[section]
\newcounter{deficislo}[section]
\newcounter{lemacislo}[section]
\newcounter{propocislo}[section]
\newcounter{exacislo}[section]
\newcounter{theorcislo}[section]

\newcommand{\pz}{P(z,\bar z)}

\newcommand{ \al}{\alpha}

\begin{document}

\maketitle

\begin{abstract} { Counterexamples  to  the  $2-$jet 
determination Chern-Moser  Theorem in  codimension $d>2$ have recently been constructed \cite{me1}.
We 
extend  the Chern-Moser approach for hypersurfaces 
 to real submanifolds of higher codimension in complex space  to derive 
results on jet determination for their automorphism  group. Using these techniques, we show     that   the  $2-$jet 
determination Chern-Moser  Theorem holds  in  codimension $2.$} 

\end{abstract}

\section{Introduction}
The local equivalence problem for real submanifolds in complex spaces is a very natural question, which was started on in complex dimension 2 by H. Poincar\'e and was then explored in the hypersurface case. Two real submanifolds $M$ and $M'$ are said to be locally equivalent at $p\in M$ and $p'\in M'$, respectively, if there exist neighborhoods $V$ of $p$ and $V'$ of $p'$, and a biholomorphic mapping $F:V\to V'$ such that $F(V\cap M)=V'\cap M'$. Local equivalence of submanifolds is obviously a very restrictive condition. Symmetrically, if such a $F$ exists, it is submitted to strong constraints, as is shown in the hypersurface case by the following classical statement due to Chern and Moser \cite{CM}: if $(M,p)$ and $(M',p')$ are smooth Levi nondegenerate real hypersurfaces, then biholomorphic germs of equivalence are uniquely determined by their jet of order 2 at point $p$.

In higher codimension, finite jet determination problems also attracted much attention. We refer in particular  to the contributions \cite{Za,BER1,BMR,eb-la-za,la-mi,ju,ju-la, mi-za} in the real analytic case, \cite{eb,eb-la,ki-za,KMZ2} in the $\mathcal{C}^\infty$ case, \cite{be-bl,be-bl-me,tu3}  in the finitely smooth case.

\medskip

Here we consider a  generic submanifold $M \subset \Bbb C^{n+d}$ of codimension  $d$, of {\it finite type} in the sense of Kohn and Bloom-Graham at a given point  $p$ \cite{BER0,BG,K}; this includes in particular the case of Levi nondegenerate hypersurfaces. 
  We will view $M$  as  a $C^r$-smooth perturbation of the  generic homogeneous submanifold $M_H$ (called the model of $M$, see Section 2 for details) associated to $M$ and given by:
\begin{eqnarray}
\label{fifi0}  M_H= \{ Im\ w= P(z, \bar z , Re \ w )\},\\
\label{corona01}  M= \{ Im\ w= P(z, \bar z , Re \ w)+ \dots \},
\end{eqnarray}  
where $P=(P_1,\dots, P_d )$  with $P_j$ a (non-zero) weighted homogeneous polynomial of degree $m_j$  with no pluriharmonic terms,  and the dots corresponding to $P_j$ being $C^r-$smooth  functions  whose derivatives of  weighted order less or equal to $m_j$ vanish.

Our first goal was to recover the $2$-jet determination theorem stated by Beloshapka in any codimension \cite{Belo88} since the original proof did not work, as we explained in \cite{bl-me1}. Our strategy was to extend the theory  of the generalized  Chern-Moser operator \cite{KMZ1, KMZ2} to smooth generic  submanifolds of higher codimension  of finite type: indeed, the kernel of the generalized Chern-Moser operator  reflects   the link between the weighted grading of  $hol(M_H,p),$  the Lie algebra of the real-analytic  infinitesimal CR automorphisms at $p$ of the model $M_H$  of $M,$  and the weighted jet determination problem for the stability group $Aut(M,p)$.

\medskip

  Using these techniques, we prove 
 that the $2$-jet determination  Chern-Moser Theorem for hypersurfaces still  holds in codimension $2.$  More precisely, we get
  
\begin{theo}\label{corona2} Let  $M \subset \Bbb C^{n+2}$ be a $C^3$-smooth generic  submanifold  of codimension $d=2$ that is  of finite type   $m=(2)$ at $0 \in M$  with a nondegenerate Levi map at $0$. Then any $h =(z+f,w+ g) \in Aut(M,0)$ is  uniquely determined by the following partial derivatives 
 \begin{itemize}
 \item the first  complex tangential derivatives $\dfrac{\partial f_j}{\partial z_k}(0), \ j,k=1, \dots, n,$  
 \item the first and second order normal derivatives $$\dfrac{\partial f_j}{\partial w_l}(0), \ \dfrac{\partial g_u}{\partial w_l}(0), \ \dfrac{\partial^2 g_u}{{\partial {w_s}}{\partial w_l}}(0),   j=1, \dots, n, \ u, l,s=1,2.$$ 
 \end{itemize}
\end{theo}

  Note that codimension 2 is a specific case, since the same techniques actually led the second author to an example of a quadric
  of codimension $5$ in $\C^9$  
  for which  
  $4$-jet determination (and not less) for biholomorphisms  holds \cite{mey}. Later on, Jan Gregorovic and the second author \cite{me1} gave  examples of quadrics with  jet determination  of arbitrarily high order  in codimension $d>2.$\\
Our main tool is given by  the following theorem   
(see  also Theorem \ref{re1} for a more explicit statement)
\begin{theo}\label{re0}{Let $M_H$ be the generic homogeneous submanifold   given by \eqref{fifi0}.  Assume that $M_H$ is holomorphically nondegenerate.
Then 
\begin{enumerate}
\item There exists   $k$ such that  $hol(M_H,0)$ admits the weighted grading
 $$hol(M_H,0) = \oplus_{\mu \le {k}} G_{\mu},\ \  G_{k} \ne \{0 \}.$$
\item For any  sufficiently finitely smooth perturbation of $M_H$  given by  \eqref{corona01},
   there exists a constant  $k_0 \ge 1$ depending only on $M_H$  such that   
 any $h \in Aut (M,0)$  is uniquely determined by  its $k+k_0$ weighted jets at $0.$ 
\end{enumerate}}
\end{theo}
We should mention   that  part (1) in Theorem \ref{re0} is an immediate consequence of  \cite{BER1}. 
  Also, an inspection of the proof of part (2), based on the Taylor expansion of the  holomorphic map $h$, shows that Theorem \ref{re0}    also  holds for formal maps sending $M$ to $M,$ where $M$ is a smooth generic submanifold (or formal submanifold) with  holomorphically nondegenerate  model $M_H$.

 \noindent  Theorem \ref{corona2} then comes from the explicit computation of $k$ for which part (1) holds in Theorem \ref{re0}. A crucial point is to use  "integrations" of a vector field (Lemma \ref{or} and Definition \ref{or1}). This notion was   introduced in \cite{KMZ1} for the hypersurface case and was the key point to get the  counterexample given in \cite{mey}.

\bigskip

 The paper is organized as follows.  In Section 2, we recall the 
notion of Bloom-Graham finite type and its basic properties. We also define  the notions of model generic submanifold  $M_H$ 
associated to $M,$   and of weighted  coordinates associated to $M_H.$  In Section 3, we show  how to 
reduce    the study of  the weighted jet determination  problem for $Aut (M,p),$  the 
stability group of $M,$ to  the study of $ hol (M_H,p),$ the set of real-analytic 
infinitesimal $CR$ automorphisms of $M_H$ at $p$ (see Theorem \ref{re1}). In Section 4, we 
introduce the notion of rigid vector fields  and  prove  results regarding the jet 
determination problem for  $ hol (M_H,p)$ (see Proposition \ref{zu1}). In Section 5, we discuss  the quadric model case  $Q$ and use  
Theorem \ref{re1} to prove   Theorem \ref{corona2} by describing $hol (M_H,p)$ when $d=2.$

\section{Preliminaries}
Usually, the study of a local CR  equivalence problem begins with the choice of appropriate coordinates, merged into bunches according to their geometric contributions, each bunch being assigned a (numerical) weight. For instance, in the case of a Levi nondegenerate hypersurface, the complex normal direction is assigned the weight $2$ while the complex tangential directions are assigned the weight $1$ (see \cite{CM}).
In the case of finite multitype in the sense of Catlin \cite{C} at a given point, the complex normal direction is assigned the weight $1$ while the complex tangential directions are assigned (possibly different) rational weights $\mu_j,$ in order to study the generalized Chern-Moser operator, as done recently \cite{CM,KMZ1,KMZ2}.

\medskip

Let  $M \subseteq \Bbb C^{n+d}$ be a   smooth   generic submanifold of real codimension  $d>1$ 
and $p \in M $ be a  point of {\it finite type} $m=(m_1, \dots,m_k),$  in the
sense of Kohn and Bloom-Graham \cite{BER0,BG,K},
where  $m_1< \dots<m_k$  are  the H\" ormander numbers.
\noindent We  consider 
local holomorphic coordinates $(z,w)$ vanishing at $p$,
where $z =(z_1, z_2, ..., z_n)$ and  $z_j = x_j + iy_j$,
$w \in \Bbb C^d$ and $w =u+iv$. Assuming  that the tangent space to
$M$ at $0$ is given by  $\{ v=0 \},$    $M$  is   described near $0$ as the graph of a uniquely
determined real vector  valued function
\begin{equation} v = \psi(z_1,\dots, z_n,  \bar z_1,\dots,\bar z_n,  u), \ d\psi(0) =0.
\end{equation}

Writing  $w=({{w}}_{m_1}, \dots, {{w}}_{m_k}),$ where ${{w}}_{m_j}$ are  vectors  of length $l_j$ (such that $d=\sum_{j=1}^k l_j m_j),$ we may assume that this $d$-dimensional equation is actually
\begin{equation}\label{fifi1}
M: \begin{cases}
v_{m_1}=\psi_{m_1}(z,\bar z,u)=P_{m_1}(z,\bar z)+\dots\\
v_{m_2}=\psi_{m_2}(z,\bar z,u)=P_{m_2}(z,\bar z, u_{m_1})+\dots\\
\vdots\\
v_{m_k}=\psi_{m_k}(z,\bar z,u)=P_{m_k}(z,\bar z, u_{m_1},\hdots,u_{m_{k-1}})+\dots\\
\end{cases}
\end{equation}
where 
$P_{m_j}(z, \bar z,  {{u}}_{m_1}, \dots, {{u}}_{{m}_{j-1}})$ are real vector valued polynomials of length $l_j$ satisfying the following  conditions of normalization
\begin{itemize}
\item  each one of the $l_j$ components of 
  $P_{m_j}$ is a homogeneous polynomial of degree $m_j$, that is, 
$$P_{m_j}(tz, t\bar z,t^{m_1}u_{m_1}, \dots , t^{m_{j-1}}u_{m_{j-1}})\equiv t^{m_j}P_{m_j}(z,\bar z,u_{m_1}, \dots , u_{m_{j-1}});$$
\item  $P_{m_j}(z, 0,  u_{m_1}, \dots, u_{m_{j-1}})\equiv 0,$ 
\item there are no terms of the form ${u_{m_k}}^{\alpha_k} \dots {u_{m_{j-1}}}^{\alpha_{j-1}}P_{m_k}$ in $P_{m_j}$ for $k<j$ (see condition (6.2.6) of Theorem (6.2) in \cite{BG}),
\end{itemize}
and the dots  terms  are sums of monomials of order strictly bigger than $m_j$ in the formal Taylor expansion of $\psi_{m_j}$.

Coordinates which provide such a description will be called {\it standard coordinates}, and $M$ given by (\ref{fifi1}) is said to be written {\it in standard form}.

\medskip

We assign natural weights to the variables: 
the tangential variables $z_1, \dots, z_n$ are given weight $\frac{1}{m_1}$ while the  component variables  of  $ {w}_{m_j}$ are given weight $\frac{m_j}{m_1}.$ 
\begin{df} {  The weighted degree $\kappa$ of a monomial $$q(z, \bar z, \dots, {u}_{m_1}, \dots, {u}_{m_k})=c_{\alpha \beta \lambda}z^{\al}\bar z^\beta {{u}_{m_1}}^{\lambda_{1}}\dots {{u}_{m_k}}^{\lambda_{k}} $$ is defined as
$$ \kappa:= 
\sum_{j=1}^k|\lambda_{j}|\frac{m_j}{m_1} +  \frac{1}{m_1} \sum_{i=1}^n (\al_i + \beta_i ).$$}\end{df} 
We obtain then the notion of weighted homogeneous polynomial
\begin{df}A polynomial $Q(z, \bar z, u)$  is weighted homogeneous
of weighted degree $\kappa$ if it is a sum of
 monomials of weighted degree $\kappa$.
 \end{df}

\begin{rem}
Note that according to this definition, $P_{m_j}$ is a vector valued weighted homogeneous polynomial of {\it weighted} degree $\frac{m_j}{m_1}$, while the dots terms   are made of {\it weighted} degree bigger than $\frac{m_j}{m_1}.$ 
\end{rem}

\begin{df}  
In this setting, we say that the generic submanifold of codimension $d$ given by 
\begin{equation} \label{fifi2} M_H = \{(z,w) \in \mathbb C^{n+d} \ | \ 
v_{m_j}=P_{m_j}(z, \bar z,  { {u}}_{m_1}, \dots, {{u}}_{m_{j-1}}),   j=1, \dots, k\} \end{equation}
is the model  submanifold of $M$ associated to the standard form (\ref{fifi1}).
\end{df}

Note that standard coordinates are not unique. For instance, in the case of $m=(m_1)$, all models are equivalent by a linear action as it is shown in the next section.

\section{The basic identities}
{In this section, $M$ is assumed to be given by (\ref{fifi1}), with $M_H$ the associated model submanifold. We follow the same approach as in \cite{KMZ1} and \cite{KMZ2}, where the hypersurface case is analyzed.}
\begin{df}{We denote by $Aut (M,0)$  the set of  germs at $0$ of 
biholomorphisms mapping $M$ into itself and fixing $0.$}\end{df}
\begin{lem} Let $h \in Aut(M,0).$ Then $h$  is of  the form
\begin{equation}\begin{aligned}
  {z}' &= z+  f(z,w)\\
   {{{w}}_{m_j}}' &= {{w}}_{m_j} +  g_{m_j}(z,w), \\
\end{aligned} \label{03} \end{equation}
where $g_{m_j}(z,w)$ (resp. $f(z,w)$) is a sum of terms of  weighted degree bigger or equal to $\dfrac{m_j}{m_1} $ (resp. bigger or equal to $\dfrac{1}{m_1}$).
\end{lem}
\begin{pf}
The statement is obvious for $f(z,w)$ and $g_{m_1}(z,w).$
Indeed, using \eqref{fifi1},  we have 
$$v_{m_1}=\psi_{m_1}(z,\bar z,u)=P_{m_1}(z,\bar z)+\dots.$$ Therefore, we obtain
$$g_{m_1}(z,w)- \overline {g_{m_1}(z,w)}=2i P_{m_1}(f,\bar f)+\dots.$$ Hence, $g_{m_1}(z,w)$ contains  no term of weight less than one.

 Suppose now that the statement is true for $g_{m_k}(z,w), \ k<j,$  and suppose by contradiction that there is in $g_{m_j}(z,w)$ a term of  minimal order $$a_{\alpha}(z) {w_{m_l}}^{\alpha_l}\dots{w_{m_k}}^{\alpha_k}$$ of   weighted degree less than  $\dfrac{m_j}{m_1}.$  If $a_{\alpha}(0) =0,$  this leads to a contradiction since one gets  a term of the form $$a_{\alpha}(z) {u_{m_l}}^{\alpha_l}\dots{u_{m_k}}^{\alpha_k}$$ of   weighted degree less than $\dfrac{m_j}{m_1},$  which is not possible using the conditions of normalization and the induction. If $a_{\alpha}(0) \ne0,$ then either we obtain a term of the form 
$$a_{\alpha}(0) {u_{m_l}}
^{\alpha_l}\dots{u_{m_k}}^{\alpha_k}$$ 
 or a term of the form $$a_{\alpha}(0)c i {u_{m_l}}^{\alpha_l -1}\dots{u_{m_k}}^{\alpha_{k}}P_{m_l},$$ 
where $c$ is a real constant. But this  can not cancel with any other term, using the conditions of normalization of $P_{m_s}, s=1 \dots k.$
 
\end{pf}

 \begin{df}{We denote  by  $hol(M,0)$ the set of  germs of real-analytic 
infinitesimal CR automorphisms of $M$ at $0.$ }\end{df}

\begin{rem}{(\cite{BER0}) Recall that $X \in hol (M_H,0)$ if and only if there exists  a germ $Z$ at $0$ of a  holomorphic vector field in $\Bbb C^{n+d}$ such that $Re Z$ is
 tangent to $M_H$ and $X= Re Z|_{M_H}.$
\noindent By abuse of notation, we also say that $Z \in  hol (M_H,0).$}\end{rem}

\noindent We  decompose the formal Taylor expansion of  $\psi_{m_j},$  denoted by 
 $\Psi_{m_j},$ into   weighted homogeneous polynomials   $\Psi_{m_j,\nu}$ of  weighted degree $\nu,$  
$$\Psi_{m_j} = \sum
\Psi_{m_j, \nu}.$$

\noindent Let $h=( {z_j}',w') 
\in Aut (M,0)$ given by \eqref{03}.

\noindent Putting   $f = (f_1, \dots, f_n),$  and $g=(g_{m_1}, \dots, g_{m_k}),$
we  consider the mapping  given by  
\begin{equation}\label{tutu}T = (f,g),\end{equation}
and, again, decompose each power series $f_j$ and $g_{m_j}$ into  weighted homogeneous polynomials $ f_{j,\mu}$ and  $ g_{m_j,\mu}$ of  weighted degree $\mu,$ 
$$f_j = \sum f_{j,\mu} , \ \ \ \ \ g_{m_j} = \sum g_{m_j,\mu}.$$

 \noindent Since  $h \in Aut (M,0),$  substituting (\ref{03}) into $v'= \psi(z', \bar z', u')$
 we obtain the transformation formula
\begin{equation}\begin{aligned}  \psi(z + f(z,u+i\psi(z, \bar z, u)),
\overline{z + f(z,u+i\psi(z, \bar z, u))}, & u \;+ \\ +\;Re\
 g
 (z,u+i\psi(z, \bar z, u))   = \psi(z, \bar z, u) +
  Im\ g(z, u+&i\psi(z, \bar z,
u)).\label{covf0}\end{aligned} \end{equation}


\noindent  Expanding (\ref{covf0}) we consider 
  terms of weight $\mu > 1$. We get
\begin{equation}
\begin{aligned} 
  2 Re \sum_{j=1}^n \ P_{m_l,z_j}(z,\bar z)  {f_{j, \ \mu-1+\frac{1}{m_1}}}(z,u +&i\pz) + \\
  2 Re \sum_{j=1}^{k-1} \ P_{m_l,w_j}(z,\bar z)  g_{m_j, \ \mu-1+\frac{m_j}{m_1}}(z,u +&i\pz)= \\
 =  Im\ g_{m_l, \ \mu -1   + \frac{m_l}{m_1}}(z,&u+i\pz) + \dots
\end{aligned}
 \label{04}
\end{equation}
where dots denote terms depending on $f_{j, \ \nu-1+\frac{1}{m_1}}, \  g_{m_j, \ \nu-1+ \frac{m_j}{m_1}},
\psi_{\nu},$ for  $\nu < \mu.$ 

\begin{prop}
Let $h=(z +f, w+g) \in Aut(M,0)$ be given by \eqref{03}. Let  
$$(f,g) =\sum_\mu (f,g)_{\mu}$$ where
 $$(f,g)_{\mu}=(f_{\mu-1+\frac{1}{m_1}}, g_{m_1, \ \mu-1+ \frac{m_1}{m_1}},  \hdots , g_{m_k,\  \mu-1+ \frac{m_k}{m_1}} ),$$ and $\mu_0$  be  
minimal such that $(f,g)_{\mu_0} \neq 0.$ 

If $\mu_0 >1,$  the (non trivial) vector field
\begin{equation}\label{fie} Y = \sum_{j=1}^n f_{j, \ \mu_0 -1+\frac{1}{m_1}} 
\frac{\partial}{\partial z_j} + \sum_{j=1}^k g_{m_j, \ \mu_0-1+\frac{m_j}{m_1}} \cdot \frac{\partial}{\partial { w}_{m_j}}
\end{equation}
lies in $hol(M_H,0),$ where $M_H$ is given by \eqref{fifi2}.
\label{th03}   \end{prop}

Here the notation $\dfrac{\partial}{\partial w_{m_j}}$ stands for the $l_j$-dimensional vector of the corresponding partial derivatives, and $g_{m_j}\cdot\frac{\partial}{\partial w_{m_j}}$ for the usual dot product.
 
 \begin{pf}
Using \eqref{04} and the definition of $\mu_0,$  we obtain 
\begin{equation}
\begin{aligned} 
  2 Re \sum_{j=1}^n \ P_{m_l,z_j}(z,\bar z)  {f_{j, \ {\mu_0}-1+\frac{1}{m_1}}}(z,u +&i\pz) + \\
  2 Re \sum_{j=1}^{k-1} \ P_{m_l,w_j}(z,\bar z)  g_{m_j, \ \mu_0-1+\frac{m_j}{m_1}}(z,u +&i\pz)= \\
 =  Im\ g_{m_l, \ \mu_0 -1   + \frac{m_l}{m_1}}(z,&u+i\pz).
\end{aligned}
 \label{05}
\end{equation}
\noindent Applying $Y$ to $v-P$ and using \eqref{05}, we obtain
\begin{equation}\label{lem0}
  Re\; Y(v_{m_l}-P_{m_l})_{|M_H} =\end{equation}
  \begin{equation}
\begin{aligned} 
  - Re \sum_{j=1}^n \ P_{m_l,z_j}(z,\bar z)  {f_{j, \ \mu_0-1+\frac{1}{m_1}}}(z,u +&i\pz) + \\
  - Re \sum_{j=1}^{l-1} \ P_{m_l,w_j}(z,\bar z)  g_{m_j, \ \mu_0-1+\frac{m_j}{m_1}}(z,u +&i\pz)- \\
   +\dfrac{1}{2}Im\ g_{m_l, \ \mu_0 -1   + \frac{m_l}{m_1}}(z,&u+i\pz)=0.
\end{aligned}
 \label{06}
 \end{equation}
 
 \end{pf}

\begin{df}{We say that the vector field $$Y = \sum_{j=1}^n F_j(z,w) \frac{\partial}{\partial z_j} + \sum_{j=1}^kG_{m_j}(z,w)\cdot\frac{\partial}{\partial { w}_{m_j} }$$
 has homogeneous weight $\mu \  (\ge -\frac{m_k}{m_1})$ if $F_j$ is  a weighted homogeneous polynomial of weighted degree $\mu+ \frac{1}{m_1},$ and $G_j$ is a homogeneous polynomial of weighted degree $\mu+\dfrac{m_j}{m_1}.$} 
\end{df}


\begin{rem}{

\noindent We write  
\begin{equation} hol(M_H,0) = \oplus_{\mu \ge -\frac{m_k}{m_1} } G_{\mu},
\end{equation} 
where $G_{\mu}$ consists of weighted homogeneous vector fields of weight $\mu.$ 
 Note  that  each weighted homogeneous component $X_{\mu}$ of $X$
 is in  $ hol(M_H, 0)$ if  $X \in hol(M_H,0).$  }
 \end{rem}

\begin{exa}\label{e1}{
The   vector fields  $\ W_{m_k, j}, j=1, \dots, l_k,$      given by
\begin{equation}\label{gentil.01}
 W_{m_k,j} = \frac{\partial}{\partial  {{w}_{m_k,j}}}
 \end{equation}
lie in $G_{-\frac{m_k}{m_1}}.$}
\end{exa}
\begin{exa}\label{e2}{The  vector field defined  by
\begin{equation}
 E =\dfrac{1}{m_1}  \sum_{j=1}^n z_j \frac{\partial}{\partial z_j } + \sum_{j=1}^k\dfrac{m_j}{m_1}{w}_{m_j}\cdot\frac{\partial}{\partial { w}_{m_j}}.
\end{equation}
 lies in $G_0.$} 
 \end{exa}

\begin{theo}\label{re1}{Let $M \subset \Bbb C^{n+d}$ be a smooth generic  submanifold  of codimension $d$ that is  of finite type  at $0$  given by  \eqref{fifi1}. Let $M_H$ be the  model hypersurface given by \eqref{fifi2}. Let $\mu_0(>\frac{m_k}{m_1})$ such that 
\begin{equation} hol(M_H,0) = \oplus_{-\frac{m_k}{m_1} \le \mu < {\mu_0-\frac{m_k}{m_1}}} G_{\mu}
\end{equation}
Then  any $h=(z +f, w+g) \in Aut(M,0)$  given by \eqref{03} such that  $(f,g)_\mu=0$ for $\mu<\mu_0$ is the identity map. }
\end{theo}

\begin{pf} 
Using  Examples \ref{e1} and \ref{e2}, we see that ${-\frac{m_k}{m_1} \le \mu < {\mu_0-\frac{m_k}{m_1}}}, $ with $\mu_0>\frac{m_k}{m_1}.$
Then we apply Proposition \ref{th03}.
 \end{pf}

 {\it Proof of Theorem \eqref{re0}}  An inspection of the proof of 
Proposition \ref{th03} shows that  the conclusion of Theorem \ref{re1} holds if $M$ is  assumed to be of class $C^{m_k+1}.$
\begin{rem}\label{ze1}
If $m=(2),$ the conclusion of Theorem \ref{re1} holds if $M$ is  assumed to be of class $C^{3}.$
\end {rem}

 \section{The components  ${G_{\mu}}$.}
 \begin{df}\label{dob} We denote by ${G_{\mu}}^R$ the set of  vector fields in  ${G_{\mu}}$ that are rigid, that is, whose coefficients depend only on $z.$
\end{df}
\begin{rem}
Note that $ W_{m_k,j}$ are rigid, while $E$ is not.
\end{rem}
We  recall the following definition
\begin{df}
A  real-analytic submanifold $M \subset \Bbb C^N$  is holomorphically nondegenerate at $p \in M$ if there is no germ  at $p$ of a holomorphic vector field $X$ tangent to $M.$
\end{df}
 \begin{prop}\label{zu1}
Let  $M \subset \Bbb C^{n+d}$ be a smooth generic  submanifold  of codimension $d$ that is  of finite type at $0$, written in standard form.   If the associated model $M_H$ is holomorphically non degenerate, then  ${G_{\mu}}^R=\{0\}$ for $\mu \ge {\dfrac{m_k-1}{m_1}}.$
\end{prop}

\begin{pf} Let $X \in {G_{\mu}}^R, \  \mu \ge {\dfrac{m_k-1}{m_1}},$ be given by 
\begin{equation}
X = \sum_{j=1}^n f_j(z) \frac{\partial}{\partial z_j} + \sum_{l=1}^k g_{m_l}(z)\cdot\frac{\partial}{\partial w_{m_l}}.
\end{equation}
We prove that $X$ itself is complex  tangent to $M_H$. First, using \eqref{fifi1} and in particular the fact that no $P_{m_r}$ contains pluriharmonic terms, we obtain that  $g_{m_l}=0.$
Then we get by assumption that for any $r=1,\hdots,k$, 
\begin{equation}
 (Re \sum_{j=1}^n f_j(z) \frac{\partial}{\partial z_j}) \left(v_{m_r} - P_{m_r}(z, \bar z,{u})\right) =0.
\label{5}\end{equation}
By the reality of $P_{m_r}$, we may rewrite (\ref{5})  as 
\begin{equation}
 Re \left (\sum_{j=1}^n f_j \frac{\partial P_{m_r}}{\partial z_j}(z, \bar z,{u})\right ) =0.
\label{6}
\end{equation}
 Write
\begin{equation}
  \sum_{j=1}^n f_j(z) \frac{\partial P_{m_r}}{\partial z_j}(z, \bar z,{u}) = \sum_{\al, \hat \al} {B_{\al \hat \al s,\, r}} z^{\al} \bar z^{\hat \al }{u^s},
\label{7}
\end{equation}
 Using (\ref{6}), we obtain 
\begin{equation}
 {B_{\al \hat \al {s},\, r}}= - \overline{{B_{\hat \al \al {s},\, r}}}.
\label{8}
\end{equation}
On the other hand, since $P_{m_r}$ is of weighted degree $\dfrac{m_r}{m_1}, $
 we have
\begin{equation}
 weight (\frac{\partial P_{m_r}}{\partial z_j}) =\dfrac{m_r}{m_1} -\frac1{m_1}.
\label{9}
\end{equation}
First we claim that ${B_{\al \hat \al {s},\, r }}$ are zero for all $\al, \hat \al, {s}, r$.
By contradiction, assume  
there is $\al , \hat \al,{s}, r, $ with ${B_{\al \hat \al {s},\, r}} \neq 0$. By assumption, 
$ |\al| \geq \dfrac{m_r}{m_1}$ whereas  $ |\hat \al | < \dfrac{m_r}{m_1}$, using (\ref{9}).
On the other hand, by (\ref{8}), we obtain that there exists a nonzero term with weight 
in $z$ less than $\dfrac{m_r}{m_1}$, and in $\bar z $ greater than or equal $\dfrac{m_r}{m_1}.$ 
That gives  a contradiction, hence all ${B_{\al \hat \al {s},\, r}}$ are zero.
Therefore we obtain that $X$ is complex  tangent to $M_H$, and since $M_H$ is holomorphically
nondegenerate,  $X=0.$
\end{pf}

   Let   $h=(z+f,w+g) \in Aut(M,0)$ be given by \eqref{03}. If    $M_H$ is holomorphically non degenerate, then there exist $N_1, N_2, N_3$ such that $g$ is uniquely determined by the following  set of derivatives \cite{BMR}
 \begin{equation}\{
 \dfrac{\partial^{|\alpha|} f}{\partial {z}^{\alpha}},\ \dfrac{\partial^{|{\beta}| + |{\gamma|}} g}{\partial {z}^{{\beta}}\partial {w}^{{\gamma}}},\ |\alpha|\le N_1,\ |{\beta}| \le N_{2},\ 0<|\gamma |\le N_{3}\}. 
 \end{equation}
 Proposition \ref{zu1} yields 
\begin{coro}
Let  $M \subset \Bbb C^{n+d}$ be a smooth generic  submanifold  of codimension $d$ that is  of finite type at $0$ written in standard form, such that the associated model $M_H$ is holomorphically non degenerate.  Let   $h=(z+f,w+g) \in Aut(M,0),$  $N_1, N_2, N_3$ as above.
  Then $N_1\le m_k-1.$ 
\end{coro}
\begin{rem}
Note that  if $m=(m_1),$ $$\{ W_{m_1,j}, j=1, \dots, d \} = {G_{-1}}^R= {G_{-1}} $$
\end{rem}
We have the following lemma whose easy proof is left to the reader.
\begin{lem}\label{or}{Let  $M \subset \Bbb C^{n+d}$ be a smooth generic  submanifold  of codimension $d$ that is  of finite type at $0$ with $m=(m_1)$, written in standard form. Let $Y \in G_{\mu}\setminus {G_{\mu}}^R $  and let $\{W_{m_1,j}, \ j=1, \dots, d\}$ be given by \eqref{gentil.01}. For every   $1 \le j \le d,$  there exist an integer  $k_j\ge 0$  and a   vector field denoted by ${\mathcal {D}}^{(k_j)}(Y) \in hol(M_H,0), \ \ ({\mathcal {D}}^{(0)}(Y)=Y),$ whose coefficients do not depend on $w_{m_1,j}$  such that $[ \dots [[Y;W_{m_1,j}];W_{m_1,j}];\dots ];W_{m_1,j}]={\mathcal {D}}^{(k_j)}(Y),$ where the string of brackets is of length  $k_j.$} 
\label{lem1}\end{lem}

This leads to the following definition in the case $m=(m_1).$
\begin{df}\label{or1}{ Let  $X \in {G_{\mu}}^R$, and $\kappa=(\kappa_1,\hdots,\kappa_d)\in\mathbb{N}^d$. We say that $Y \in {G_{\mu+\sum_{j=1}^d k_j}}$  is a  $\kappa$-integration of  $X$  if 
\begin{equation}{\mathcal {D}}^{(\kappa_1)}(\dots({\mathcal {D}}^{(\kappa_d)}(Y)\dots)= X.
\end{equation}
 We denote $Y$ by ${\mathcal{D}}^{-(\kappa)}(X).$}
\end{df}

By abuse of notation, we will also refer to ${\mathcal{D}}^{-k}(X)$ where $k=|\kappa|=\sum_{i=1}^d\kappa_i$, since the resulting integrated vector fields will be treated similarly.

 \section{The components  ${G_{\mu}}$ for the quadric.}

We wish to discuss  Theorem \ref{re1} in  the case of a smooth generic  submanifold  of codimension $d$ that is  of finite type $m=(2)$. Once written in standard form, we get that the model submanifold is a quadric $\mathcal{Q}$, that is,
\begin{equation} \label{boston} \mathcal{Q}=
\begin{cases}
v_1={\ }^t\bar{z} A_1 z,\\
v_2={\ }^t\bar{z} A_2 z,\\
\dots,\\
v_j={\ }^t\bar{z} A_j z,\\
\dots,\\
v_d={\ }^t\bar{z} A_d z\\
\end{cases}
\end{equation} 
with  $A_j, j=1,\dots, d$  being  linearly independent Hermitian matrices.
 
\begin{rem}\label{vinc}
In coordinates, $A=(A_1,\hdots,A_d)$ corresponds to the Levi map of $M$ at 0. The linear independence of the $A_i$ is actually equivalent to being of finite type $m=(2)$ at 0 for $M \subset \Bbb C^{n+d}$ a smooth generic  submanifold  of codimension $d$.
\end{rem}

The properties of the standard form also give the following lemma.
\begin{lem} \label{mah2}
Let  $M \subset \Bbb C^{n+d}$ be a smooth generic  submanifold  of codimension $d$ that is  of finite type at $0$, written in standard form with its model quadric given by   \eqref{boston}.   Then
\begin{equation}\label{mah1}\cap \ker{(A_j)} =\{0\} \end{equation} if and only if there is no holomorphic tangent vector field to $\mathcal{Q}$.
\end{lem}
A direct application of Proposition \ref{zu1} yields 
\begin{coro}\label{zu}
If the conditions in Lemma \ref{mah2} are satisfied, then  ${G_{\mu}}^R=\{0\}$ for $\mu >0.$
\end{coro}

\begin{rem} Note that
 ${G_{0}}^R\ \ne \{0\},$ since  the vector field 
\begin{equation}\label{stres}
\sum_{j=1}^n iz_j \dfrac {\partial}{\partial z_j} \in {G_{0}}^R.
\end{equation}
\end{rem}
\begin{rem}
We also get that ${\mathcal{D}}^{-1}({G_0}^R)\not=\{0\}$ can only happen if the codimension $d$ is bigger than 2, since for $d=1$ Chern Moser's Theorem \cite{CM} shows that  the mixed derivatives  $\dfrac{\partial^2 f_j}{{\partial {w}}{\partial {z_k}}}, \ j,k=1, \dots, n,$ are  not needed. 
And it actually happens, as in the following example.
\end{rem}

\begin{exa}\label{exa0} Let $M$ be given by
{\begin{equation*}
M=\{ (z_1, z_2, z_3, z_4, w_1, w_2, w_3)) \in \Bbb C^7 \ | \ \end{equation*} \begin{equation*}v_1=z_3 \bar {z_3}, \  v_2=z_4 \bar {z_4},\ v_3=z_1 \bar {z_3} + z_3 \bar {z_1} +z_2 \bar {z_4}+z_4 \bar{ z_2}.\}
\end{equation*}

The associated matrices are
$$ A_1=
\begin{pmatrix}
0&0&0&0\\
\,\,0&0&0&0\\
\,\,0&0&1&0\\
\,\,0&0&0&0\
\end{pmatrix}
 A_2=
\begin{pmatrix}
0&0&0&0\\
\,\,0&0&0&0\\
\,\,0&0&0&0\\
\,\,0&0&0&1\
\end{pmatrix}$$$$
 A_3=
\begin{pmatrix}
0&0&1&0\\
\,\,0&0&0&1\\
\,\,1&0&0&0\\
\,\,0&1&0&0\
\end{pmatrix}$$}

Let $X$ and $Y$ be  the vector fields 
$$
X=iz_3\frac{\partial}{\partial z_1}, \ Y=-iz_4\frac{\partial}{\partial z_2}$$

It is easy to check that
$$
w_1Y +w_2 X \in {\mathcal{D}}^{-1}({G_0}^R)
$$
\end{exa}
\begin{theo}\label{bre1}
Let  $M \subset \Bbb C^{n+2}$ be a smooth generic  submanifold  of codimension $d=2$ that is  of finite type   $m=(2)$ at $0$, written in standard form. Assume the associated model quadric $\mathcal{Q}$ is holomorphically non degenerate, then 
\begin{itemize}
\item $(i)$ ${\mathcal{D}}^{-1}({G_0}^R)\setminus ({\mathcal{D}}^{-1}({G_0}^R) \cap{\mathcal{D}}^{-2}({G_{-1}}^R))  = \{0\}$
\item $(ii)$ ${\mathcal{D}}^{-2}({G_{-\frac{1}{2}}}^R)  = \{0\}$
\item $(iii)$ ${\mathcal{D}}^{-1}({G_{-1}}^R) \ne 0, \ \ {\mathcal{D}}^{-3}({G_{-1}}^R)  = \{0\}$
\end{itemize}
\end{theo}

Before proving Theorem \ref{bre1}, let us explain how it leads to theorem \ref{corona2}, providing a generalization of the $2$-jet determination Chern-Moser Theorem in the case of codimension $2.$

According to Lemma \ref{or}, we need to study  the  $\kappa$-integrations of any rigid vector field.  Theorem \ref{bre1} provides  the precise  $\kappa$-integrations needed, and  shows that at most  $2$-integrations are needed, depending on  the rigid vector field. Using  Theorem \ref{re1} and Remark \ref{ze1}, we  then conclude that any $h =(z+f,w+ g) \in Aut(M,0)$ is  uniquely determined by the following partial derivatives 
\begin{itemize}
 \item the first  complex tangential derivatives $\dfrac{\partial f_j}{\partial z_k}(0), \ j,k=1, \dots, n,$    corresponding to $(i)$ in Theorem \ref{bre1},
 \item the first and second order normal derivatives $$\dfrac{\partial f_j}{\partial w_l}(0), \ \dfrac{\partial g_u}{\partial w_l}(0), \ \dfrac{\partial^2 g_u}{{\partial {w_s}}{\partial w_l}}(0),   j=1, \dots, n, \ u, l,s=1,2$$   corresponding  to $(ii)$ and $(iii)$ in Theorem \ref{bre1}.
 \end{itemize}  
 
\begin{proof} In this special case of codimension $d=2$ with $m=(2)$, we refer to (\ref{boston}) by setting $Q=(Q_1,Q_2)$ where $Q_i={\,}^t\bar{z}A_iz$ ($i=1,2$), instead of $P_{m_1},P_{m_2}$.
\begin{itemize}
\item {\it Proof of $(i)$ -} Suppose by contradiction that there exist $X, Y \in {G_{0}^R}$  such that
\begin{equation}\label{dem}
w_1X +w_2 Y \in {\mathcal{D}}^{-1}({G_0}^R).
\end{equation}
Without loss of generality, we may assume that $[X,Y]=0,$ since otherwise, we have
$w_2 [X,Y] \in {\mathcal{D}}^{-1}({G_0}^R),
$ which is not possible unless $[X,Y]=0.$
Then, by assumption, we obtain the following equation
\begin{equation}
\label{dem1}Q_1X(Q) + Q_2 Y(Q)=0.
\end {equation}
Using the fact that $M$ is of finite type $m=(2),$  \eqref{dem1} leads to \begin{equation}
X(Q)=\alpha Q_2, \ Y(Q)=\beta Q_1,\end{equation}  $ \alpha$ and  $\beta$ complex valued vectors.

Using the fact that $ [X,Y]=0,$ and the fact that the model quadric $\mathcal{Q}$ is holomorphically nondegenerate, we obtain that $X=Y=0.$

\item {\it Proof of $(ii)$ -} 
Let $Z_1 \in  {G_{-\frac{1}{2}}}^R$ be of the form
\begin{equation}
Z_1=a_1 \dfrac{\partial}{\partial z_1} +a_2 \dfrac{\partial}{\partial z_2}+b_1(z) \dfrac{\partial}{\partial w_1}+b_2(z) \dfrac{\partial}{\partial w_2}, 
\end{equation}
with $ a_i \in \Bbb C,$  and $ b_i(z)$ linear.
If  $Z_1$  integrates,  we obtain, after a possible permutation of the variables $w_1$ and $w_2,$ an equation of the form
 \begin{equation}\label{stresa2}
  Q_1 Im \  Z_1 (Q) +   Q_2 Im \  Z_2 (Q) +\ Re X_1 (Q)=0,
 \end{equation}
 where  $Z_2 \in {G_{-\frac{1}{2}}}^R$ and $X_1$ is a vector field of weight $\frac{1}{2}.$
If $ {\mathcal{D}}^{-2}(Z_1)$ exists, then we obtain the following system
 \begin{equation}\label{stresa1}
\begin{cases}
  Q_1Im \ X_1(Q) + Q_2Im \ X_2(Q)  
 =0\\
 \\
  Q_1Im \ X_2(Q) + Q_2Im \ X_3(Q) 
 =0,\\
 \end{cases}
 \end{equation}
 where   $X_j$ are vector fields of weight $\frac{1}{2}.$
 It is not hard to see that since $Q_1$ and $Q_2$ are linearly independent Hermitian forms, the only solution to the system \eqref{stresa1} is  the trivial solution. Hence, using  Corollary \ref{zu} and  \eqref{stresa2},  the following system of equations holds
 \begin{equation}\label{stresa3}
\begin{cases}
  Q_1 \ Z_1(Q) + Q_2 \ Z_2(Q)  
 =0\\
 \\
  Q_1 \ Z_2(Q) + Q_2 \ Z_3(Q) 
 =0,\\
 \end{cases}
 \end{equation}
 where  $Z_3 \in {G_{-\frac{1}{2}}}^R.$
 Using \eqref{stresa3}, we conclude that $Z_1(Q)=0,$  and hence $Z_1=0,$ which gives the contradiction.
 
\item {\it Proof of $(iii)$ -} By integrating $W_1= \dfrac{\partial}{\partial w_1},$ we obtain an equation of the form 
\begin{equation}\label{ex3}
 AQ + Re X (Q) =0,
\end{equation}
where $A$ is a nonzero $2\times2$ real matrix, and 
 $X$ is a  vector field of weight $0.$
 Using the Euler field, we conclude  that \eqref{ex3} holds (with $A=-I$). Hence ${\mathcal{D}}^{-1}({G_{-1}}^R) \ne 0.$
 If ${\mathcal{D}}^{-3}({G_{-1}}^R) \ne 0,$ we obtain  a nontrivial equation of the form
 \begin{equation}\label{alex}             
 \sum c_{\alpha_1\alpha_2}{Q_1}^{\alpha_1}{Q_2}^{\alpha_2} =0.
 \end{equation}
 But \eqref{alex} is not possible by assumption of finiteness. Indeed, \eqref{alex} would  imply  $Q_1 = \alpha Q_2,$  which contradicts that  $M$ is of finite type  $m=(2).$ 
 \end{itemize}
 
\end{proof}

\end{document}